\documentclass[11pt]{amsart}
\usepackage{fullpage,verbatim,amssymb,mathtools}
\usepackage{hyperref,subfig}
\usepackage[active]{srcltx}
\usepackage[usenames,dvipsnames]{color}
\usepackage[normalem]{ulem}
\usepackage{cancel}
\usepackage{mathabx}
\makeatletter
\let\@@pmod\mod
\DeclareRobustCommand{\mod}{\@ifstar\@pmods\@@pmod}
\def\@pmods#1{\mkern4mu({\operator@font mod}\mkern 6mu#1)}
\makeatother

\definecolor{blue}{rgb}{0,0,1}
\definecolor{red}{rgb}{1,0,0}
\definecolor{green}{rgb}{0,.6,.2}
\definecolor{purple}{rgb}{0.75,0,0.75}
\definecolor{teal}{rgb}{0,0.65,0.65}

\long\def\red#1\endred{\textcolor{red}{#1}}
\long\def\blue#1\endblue{\textcolor{blue}{#1}}
\long\def\purple#1\endpurple{\textcolor{purple}{ #1}}
\long\def\green#1\endgreen{\textcolor{green}{#1}}
\long\def\teal#1\endteal{\textcolor{teal}{#1}}

\newcommand{\sm}{\left(\begin{smallmatrix}}
\newcommand{\esm}{\end{smallmatrix}\right)}
\newcommand{\bpm}{\begin{pmatrix}}
\newcommand{\ebpm}{\end{pmatrix}}

\newcommand{\ZZ}{\mathbb{Z}}
\newcommand{\QQ}{\mathbb{Q}}
\newcommand{\RR}{\mathbb{R}}
\newcommand{\NN}{\mathbb{N}}
\newcommand{\CC}{\mathbb{C}}
\newcommand{\frakM}{\mathfrak{M}}
\newcommand{\frakm}{\mathfrak{m}}
\newcommand{\sign}{\textup{sign}}

\newcommand{\Imag}{\textup{Im}}

\newcommand{\arcsinh}{\textup{arcsinh}}
\newcommand{\eps}{\varepsilon}

\DeclareMathOperator{\1}{\mathbf{1}}

\newtheorem{theorem}{Theorem}
\newtheorem{lemma}[theorem]{Lemma}
\newtheorem{proposition}[theorem]{Proposition}

\theoremstyle{remark}

\numberwithin{theorem}{section}
\numberwithin{equation}{section}
\renewcommand{\phi}{\varphi}
\renewcommand{\Re}{\operatorname{Re}}

\newcommand{\Err}{\operatorname{Err}}

\title{Murmurations of Maass forms}

\author{Andrew R. Booker}
\author{Min Lee}
\author{David Lowry-Duda}
\author{Andrei Seymour-Howell}
\author{Nina Zubrilina}

\thanks{The first author is supported by the Heilbronn Institute for Mathematical Research.
The second author is supported by a Royal Society University Research Fellowship.
The third author is supported by the Simons Collaboration in Arithmetic
Geometry, Number Theory, and Computation via the Simons Foundation grant
546235. The fourth author is supported by the Royal Society. The last author is supported by the Hertz Foundation and the National
Science Foundation.}
\begin{document}

\begin{abstract}
We prove the existence of murmurations in the family of Maass forms of 
weight $0$ and level $1$ with their Laplace eigenvalue parameter going to infinity (i.e., correlations between the parity and Hecke eigenvalues at primes growing in proportion to the analytic conductor).
\end{abstract}

\maketitle

\section{Introduction}
``Murmurations'' in families of automorphic forms, first observed by He, Lee, Oliver, and Pozdnyakov~\cite{HLOP} and later in~\cite{sutherland}, are a correlation between root numbers of $L$-functions and their Dirichlet coefficients.
Zubrilina~\cite{zubrilina} and Bober, Booker, Lee, and Lowry-Duda~\cite{bober2023murmurations} have given theoretical confirmation for this bias in archimedian and non-archimedian families of holomorphic modular forms.
So far, the families that have been studied (elliptic curves, modular forms, and Dirichlet characters) are arithmetic in nature.
This raises the question: does the phenomenon also take place in non-arithmetic settings?

In this work, we demonstrate this phenomenon in the family of weight $0$ level $1$ Maass forms ordered by analytic conductor (eigenvalue). Since this family of automorphic forms does not have an arithmetic analogue, it suggests that murmurations are an analytic phenomenon that can occur in families of $L$-functions coming from non-arithmetic objects.

To state our result, let $\{f_j\}$ denote a Hecke eigenbasis of weight $0$ Maass cusp forms of level $1$ with Laplace eigenvalues $\lambda_j = \frac14 + r_j^2$. 
Without loss of generality, we assume $r_j \geq 0$.
For an eigenform $f_j$, denote its Hecke eigenvalues by $a_j(n)$ for $n \neq 0$, and let $\epsilon_j = \epsilon(f_j)$ be the parity of the form, i.e., $\epsilon_j = 1$ if $f_j$ is even, and $\epsilon_j = -1$ if $f_j$ is odd. Note that $\epsilon_j$ coincides with the root number of the $L$-function of $f_j$.
Define the analytic conductor of a Maass form $f$ with eigenvalue $\lambda = \frac14 + R^2$ by
\begin{align}\label{conductor}
\mathcal{N}(R)
:= \frac{\exp \left( \psi \left( \frac{1/2+a+iR}{2} \right)
+ \psi \left( \frac{1/2+a-iR}{2} \right) \right)}{\pi^2}
= \frac{R^2}{4 \pi^2} + O(1), 
\end{align}
where $a = 0$ if $f$ is even and $a=1$ if $f$ is odd. We prove the following:
\begin{theorem}\label{mainthm}
Assume GRH for $L$-functions of Dirichlet characters and Maass forms. 
Let $E \subset \RR_+$ be a fixed compact interval with $\lvert E \rvert > 0$.
Let $R,H \in \RR_{>0}$ with $R^{\frac{5}{6} + \delta} < H < R^{1 - \delta}$ for some $\delta > 0$ 
and let $N = \mathcal{N}(R)$.
Then as $R \to \infty$, we have
\begin{equation}\label{eqn:maineqn}
\frac{\sum_{\substack{p\,\mathrm{prime}\\p/N\in E}}\log{p}\sum_{|r_j-R|\le H}\epsilon_j a_j(p)}
{\sum_{\substack{p\,\mathrm{prime}\\p/N\in E}}\log{p}\sum_{|r_j-R|\le H}1}
= \frac{1}{\sqrt{N}}\frac{\nu(E)}{\lvert E \rvert} + o_{E}\bigg(\frac{1}{\sqrt{N}}\bigg),
\end{equation}
where
\begin{align}
\nu(E) = \frac{1}{\zeta(2)} \sum_{\frac{q^2}{a^2} \in E}^*
\frac{\mu(q)^2}{ \phi(q)^2 \sigma(q)}\left(\frac{a}{q}\right)^{-3}.
\end{align}
Here the $*$ indicates the terms occuring at the endpoints of $E$ are halved. 
\end{theorem}
A plot of $\nu(E)$ against some numerically computed trace formula values is given in Figure \ref{Fig:nu}.

\begin{figure}
\includegraphics[width=\textwidth]{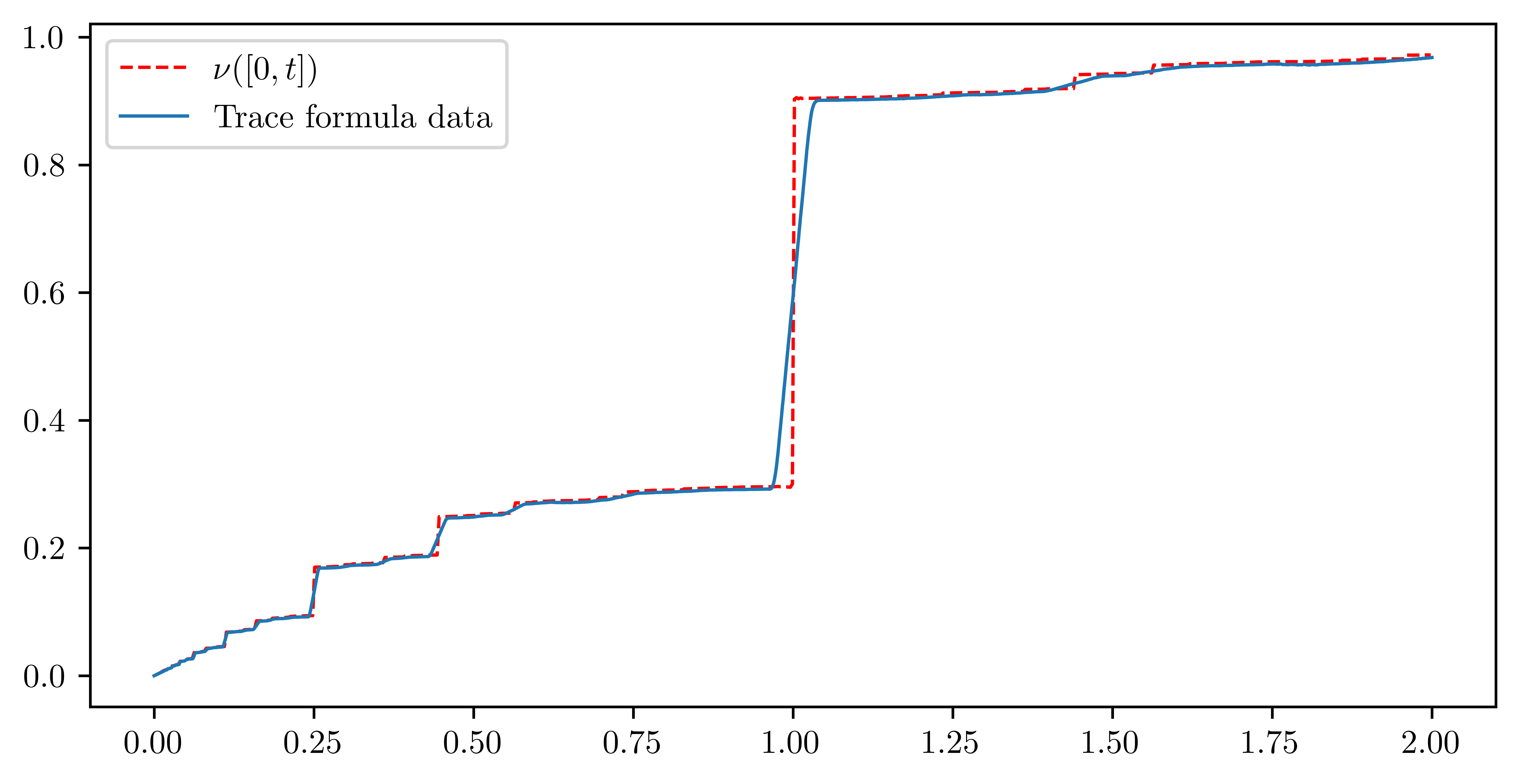}
\caption{Plot of $\nu([0,t])$ and the left hand side of \eqref{eqn:maineqn} scaled by $t \sqrt{N}$, for $R = 6900, H = 100$ and $t \in [0,2]$. The left hand side was computed using the formula \eqref{eq:traceformula}.}
\label{Fig:nu}
\end{figure}
The fact we get the same result to the holomorphic case in the weight aspect is due the relation between the weight of holomorphic forms and the Laplace eigenvalue of Maass forms. One way to see this relation is that when a weight $k$ holomorphic form is scaled by $y^{\frac{k}{2}}$, it becomes a Maass form with (weight $k$) Laplace eigenvalue $\frac{k}{2}(1 - \frac{k}{2})$ (see \cite[Exercise 2.17]{bump}).

The proof of the theorem closely follows the techniques in the paper \cite{bober2023murmurations}. The main differences in the proof stem from the differences in the trace formulae between holomorphic and non-holomorphic forms. Here, we use an explicit version of the Selberg trace formula due to Str\"{o}mbergsson \cite{Andreaspre}. The first main difference comes from the fact we need an analytic test function on the spectral side, meaning it cannot be compactly supported like in the holomorphic case. It is here we use GRH for Maass forms to control the cutoff error of using this test function to approximate the interval function. This could be removed with a smoother version, albeit with considerable more work. 

Secondly, the trace formula itself looks rather different and includes several more terms. Once we fix our choice of the test function in Section \ref{subsec:testfunction}, most of these terms are taken into the error terms, ultimately leaving us with a finite sum that will give the main term.

Lastly, similar to the holomorphic case, this sum will include $L(1,\psi_{D})$ values of quadratic fields, although with discriminant $t^2+4n$ as opposed to $t^2-4n$. This occurs due to the fact that the inclusion of the root number in the spectral side means we are actually working with $T_{-n}$ Hecke operators. Fortunately, since we replace the $L$-function values by their averages, the local analysis from \cite[Sec. 4]{bober2023murmurations} remains the same.

The overview of the proof of the holomorphic case given in \cite[Sec 2.1]{bober2023murmurations} can also be used to give an overview of the proof given here, albeit with the caveats noted above. 
Furthermore, most of remarks given in the introduction of \cite{bober2023murmurations} are also relevant here, again with the weight $K$ replaced by the Laplace eigenvalue $R$.

Explicit trace formula for Maass cusp forms have been worked out for squarefree level, also by Str\"{o}mbergsson. 
However, these formulas do not include the root number in the spectral side. 
One would need to derive an explicit formula which also includes the eigenvalue of the Fricke involution in the spectral side, but for non-holomorphic forms.

\section{Selberg Trace Formula} \label{sec:STF}

The Selberg trace formula was introduced by Selberg \cite{Sel} and intensively studied in two volume books \cite{Hej1, Hej2} by Hejhal. Here we begin by stating an explicit version of the Selberg trace formula from an unpublished result of Str\"{o}mbergsson \cite{Andreaspre}. It appears in \cite[Thm. 2.2.1]{andreithesis} and can be obtained from \cite[Prop. 2.1]{AndyMin}.

To state the Selberg trace formula we introduce the following notation. 
For $D=d\ell^2$, $d$ is a fundamental discriminant and $\ell\in \NN$, we define 
$\psi_D(n)=\left(\frac{d}{\frac{n}{\gcd(n, \ell)}}\right)$ for $n\in \ZZ$, by the Kronecker symbol. 
When $\ell=1$, $\psi_d(n)$ is a quadratic character modulo $d$. 
Then 
\begin{equation}
L(s, \psi_D) = \sum_{n=1}^\infty \frac{\psi_D(n)}{n^s}
\end{equation}
has analytic continuation to $s\in \CC$ and at $s=1$, 
\begin{equation}
L(1, \psi_D) = L(1, \psi_d) \frac{1}{\ell}\prod_{\substack{p\mid \ell\\ p^{\alpha}\|\ell}} \bigg[1+(p-\psi_d(p)) \frac{p^{\alpha}-1}{p-1}\bigg].
\end{equation}
We let $\sigma_1(n) = \sum_{d\mid n} d$ be the divisor function and $\Lambda$ be the von Mangoldt function.

\begin{theorem}[The Selberg trace formula for Maass newforms for level $1$]\label{thm:traceformula}

Fix $\delta > 0$, let $F(t)$ be an even analytic function on the strip $\lbrace t \in \CC : \lvert \Imag(t) \rvert \leq \frac{1}{4\pi} + \delta \rbrace$ such that $F(r) = O((1 + |r|^2)^{-1-\delta})$. 
Define $G$ as the Fourier transform of $F$ given by
\begin{equation}
G(u) =  \int_{-\infty}^{\infty} F(r) e^{-2 \pi iru}\,dr.
\end{equation}
Let $\lbrace f_j \rbrace$ be a sequence of normalised Hecke eigenforms of level $1$, with Laplacian eigenvalues $\lambda_j = \frac{1}{4} + r_j^2$ and respective Hecke eigenvalues $a_j(n)$ for any non-zero integer $n$. 
Then
\begin{align*}
&\frac{\sigma_1(|n|)}{\sqrt{|n|}}F \left(\frac{i}{4 \pi}\right) +  \sum_{j > 0} F\left(\frac{r_j}{2 \pi}\right) a_j(n)\\
&= \underset{\substack{t \in \ZZ\\\sqrt{D} = \sqrt{t^2 - 4n} \not\in \QQ}}{\sum} L(1,\psi_D) \cdot \begin{dcases}
G \Biggl( \log \biggl( \frac{(|t| + \sqrt{D})^2}{4 |n|} \biggr)\Biggr) & \text{if } D > 0,\\
\frac{\sqrt{\lvert D/4n \rvert}}{2 \pi} \int_{-\infty}^{\infty} \frac {G(u) \cosh(u/2)}{\sinh^2(u/2) + |D/4n|}\,du & \text{if } D < 0  \end{dcases}\\
&\quad + \underset{\substack{ad = n \\ a>0 \\ a \neq d}}{\sum} \left( \log{\pi} + \log{|a-d|} - \frac{\log(\eta(|a-d|))}{|a-d|} \right) \cdot G\left( \log \left| \frac{a}{d} \right| \right) \\
&\quad+ \frac{1}{2} \underset{\substack{ad = n \\ a>0 \\ a \neq d}}{\sum}
\int_{\bigl\lvert \log \lvert\frac{a}{d}\rvert \bigr\rvert}^{\infty} G(u) \cdot \frac{e^{u/2} + \eps e^{-u/2}}{e^{u/2} - \eps e^{-u/2} + \left| \sqrt{|a/d|} - \eps \sqrt{|d/a|} \right|}\,du\\
&\quad+ \underset{\substack{ad = n \\ a>0}}{\sum} \biggl[ G \left( \log \left| \frac{a}{d} \right| \right) \log(4e^{\gamma}) + \int_{0}^{\infty} \frac {G(u + \log \left| \frac{a}{d} \right|) - G(\log|\frac{a}{d}|)}{2 \sinh(u/2)}\,du - \frac{1}{4} F(0)\biggr]\\
&\quad+2\sum_{m=2}^{\infty} \underset{\substack{ad = n \\ a>0}}{\sum} \frac{\Lambda(m)}{m} G\left( \log\left| \frac{a}{d} \right| - 2\log{m} \right)\\
&\quad+\begin{dcases}
\Bigl[ -\frac{1}{12 \sqrt{n}} \int_{-\infty}^{\infty} \frac{G'(u)}{\sinh\left( \frac{u}{2} \right)}\,du + \left(\log\left( \frac{\pi \sqrt{n}}{2} \right) + \gamma \right) G(0) \\- \int_{0}^{\infty} \log\left(2 \sinh \left( \frac{u}{2} \right)\right) G'(u)\,du \Bigr]
& \textup{ if } \sqrt{n} \in \ZZ,
\\ 0 & \textup{ otherwise.}
\end{dcases}
\end{align*}
Here $\eps = \sign(n)$ and $\eta(m) = \prod_{k \bmod m} \gcd(k,m)$.
\end{theorem}

We apply the trace formula for $n = - p$ for a prime $p$, using that
\begin{equation*}
a_j(-p) = \epsilon_j a_j(p)
\end{equation*}
where $\epsilon_j\in \{-1, 1\}$ is the eigenvalue of the involution $f_j(-\bar{z}) = \epsilon_j f_j(z)$. 
Note that $\sqrt{t^2 + 4p} \in \QQ$ if and only if $4p = (t - m)(t + m)$ for some $m \in \NN$.
Since the two multiplicands have the same parity, this can only happen when $t = \pm( p-1)$.
Note also that as $F$ is even, $G$ is even and $G(-\log p) = G(\log p)$.
The trace formula becomes
\begin{equation}\label{eq:traceformula}
\begin{split}
&\frac{p+1}{\sqrt{p}}F \left(\frac{i}{4 \pi}\right) + \sum_{j > 0} F\left(\frac{r_j}{2\pi}\right) \epsilon_j a_j(p) \\
&= \underset{\substack{t \in \ZZ\\ D := t^2 + 4p \\  t \neq \pm (p-1 )}}{\sum} L(1,\psi_D)
       G \Biggl( \log \biggl( \frac{(|t| + \sqrt{D})^2}{4 p} \biggr)\Biggr) \\
    &\quad+ 2\left( \log{\pi} + \log(p+1) - \frac{\log(\eta(p+1))}{p+1} \right) \cdot G\left( \log{p} \right)  \\
    &\quad+ \int_{\log{p}}^{\infty} G(u) \cdot \frac{e^{u/2} -  e^{-u/2}}{e^{u/2} +  e^{-u/2} +  \sqrt{p} +\sqrt{1/p}}\,du \\
    &\quad+2\log(4e^{\gamma}) G(\log{p}) - \frac{1}{2} F(0)  \\
    &\quad+ \int_{0}^{\infty} \frac {G(u + \log{p}) - G(\log{p})}{2 \sinh(u/2)}\,du + \int_{0}^{\infty} \frac {G(u - \log{p}) - G(\log{p})}{2 \sinh(u/2)}\,du  \\
    &\quad+2\sum_{m=2}^{\infty} \frac{\Lambda(m)}{m} \bigl(G(\log{p} - 2\log{m} ) + G(\log{p} + 2\log{m}) \bigr).
  \end{split}
\end{equation}

\subsection{Choice of test function} \label{subsec:testfunction}
In this section, we discuss the choice of the test function $F$ that we will use in~\eqref{eq:traceformula}.
To apply the trace formula, we approximate the sharp cutoff characteristic function of the interval $[R - H, R + H]$ appearing in Theorem \ref{mainthm} with a smoothed function $F$ constructed using a result of Ingham:

\begin{theorem}[\cite{Ingham}]
There exists an even and non-negative entire Schwartz function $W$ such that $\widehat{W}(0) = 1$, $\widehat{W}$ is compactly supported on $[-1,1]$ and decreasing on $[0,1]$, and as $x \to \infty$,
\begin{equation}\label{e:Ingham}
W(x) = O\left(\exp\Big(\frac{-|x|}{\log^2(2 + |x|)}\Big)\right).
\end{equation} 
\end{theorem}

We will use the following tail estimate for $W(x)$.

\begin{lemma}\label{lem:W_tail_estimate}
For $W$ as above and $T \geq 0$,
\begin{equation}\label{e:tail_Err}
\int_T^\infty W(x) \, dx 
\ll \Err(T)
:= \exp \Big(\frac{-T}{5 \log^2(T + 2)}\Big).
\end{equation}
\end{lemma}

\begin{proof}
For sufficiently large $x>0$ we have
\[\frac{x}{\log^2{(2 + x)}} > \sqrt{x} + \tfrac12\log{x}.\]
Since $x+2\le(T+2)^2$ for all $x\in[T,T^2]$, for our choice of $W$ and any sufficiently large $T$, by \eqref{e:Ingham}, 
\[\int_T^\infty W(x)\,dx
\ll \int_T^{T^2} e^{-\frac{x}{4\log^2 (T + 2)}}\,dx 
+ \int_{T^2}^\infty e^{-\sqrt{x}}x^{-\frac12}\,dx
\ll \log^2{(T+2)} e^{-\frac{T}{4\log^2 (T+2)}}
\ll e^{-\frac{T}{5\log^2(T+2)}}.\]
Since $W$ is integrable,
\[\int_T^\infty W(x)\,dx \ll e^{-\frac{T}{5\log^2(T+2)}}\]
for all $T \geq 0$.
\end{proof}

Let $h > 1$ be a parameter such that $h = o(H)$ and let
\[W_h(x) = \frac{1}{h} W\left(\frac{x}{h}\right).\]
We define
\[F:=\big(\1_{[\frac{-R-H}{2\pi},\frac{-R+H}{2\pi}]} 
+ \1_{[\frac{R-H}{2\pi},\frac{R+H}{2\pi}]}\big) \ast W_h\]
as our test function, where $\1_I$ denotes the characteristic function of the interval $I$ and $\ast$ denotes the convolution. We symmetrize the function to include intervals around both $-R$ and $R$ in order for $F$ to satisfy the assumptions of Theorem \ref{thm:traceformula}.
The Fourier transform of $F$, normalized as in the trace formula in Theorem \ref{thm:traceformula}, is given by
\begin{equation}\label{e:G_choice}
G(t)= \int_{\RR} F(r) e^{-2\pi irt} \, dr
= 2\frac{\cos( R t) \sin( H t)}{ \pi t} \widehat{W}(t h).
\end{equation}
Observe that $G$ is supported on $[-1/h, 1/h]$.

We are now ready to bound the error that comes from approximating the characteristic
function for $\lvert r_j - R \rvert < H$ by $F$.

\begin{lemma}\label{smoothingerror}
Assume GRH for the $L$-functions for Maass cusp forms $L(s,f_j)$. 
Let $E\subset \RR_+$ be a fixed compact interval with $|E|>0$.
Assuming that $R - H > h$, we have
\[\sum_{\substack{p\,\mathrm{prime}\\p/N\in E}}
\log{p} \sum_{\lvert r_j-R \rvert \le H} \epsilon_j a_j(p)
= \sum_{\substack{p\,\mathrm{prime}\\p/N\in E}}
\log{p} \sum_{j > 0} F(r_j/ 2 \pi) \epsilon_j a_j(p)
+ O \big( R^{1 + \varepsilon} h \big).\]
\end{lemma}

\begin{proof}
Define
\[I_\pm = \left\{t\in\mathbb{R}:\left|t-\frac{\pm R}{2\pi}\right|\le\frac{H}{2\pi}\right\} \quad \text{ and } 
\quad F^{\pm} = \1_{I_\pm} * W_h.\]
Then $F=F^++F^-$.
The behavior of $F^{+}(t/2 \pi)$ depends on whether $t \in [R - H, R + H]$,
and also on the distance from $t$ to this interval.
To track this distance, we define
\[\xi = \xi(t; R, H) 
:= \frac{1}{h} \min \big\{%
\lvert t - (R + H) \rvert,
\lvert t - (R - H) \rvert
\big\}.\]
By definition,
\[
F^+(t/2 \pi)
= \int_{(R - H)/2 \pi}^{(R + H)/2 \pi}
W_h(t/2 \pi - \tau)\, d\tau
= \int_{(R - H - t)/2 \pi h}^{(R + H - t)/2 \pi h} W(\tau) \, d \tau.
\]
For $t \in [R-H, R+H]$, we have the main term estimate
\[
F^+(t/2\pi)
= \int_{\mathbb{R}} W(\tau) \, d \tau
- \int_{-\infty}^{(R - H - t)/2 \pi h} W(\tau) \, d\tau
- \int_{(R + H - t)/2 \pi h}^{\infty} W(\tau) \, d\tau
= 1 + O(\Err(\xi)),\]
where we have used $\widehat{W}(0) = 1$ and the tail estimate in Lemma~\ref{lem:W_tail_estimate}.
If $t > R + H$ or $t < R - H$, we use the naive bounds
\begin{align*}
F^+(t/2\pi)
&\leq \int_{-\infty}^{(R + H - t)/2 \pi h} W(\tau)\, d\tau
= O(\Err(\xi)) 
\qquad (\text{when } t > R + H),
\\ F^+(t/2\pi) 
&\leq \int_{-\infty}^{(R - H - t)/2 \pi h} W(\tau)\, d\tau
= O(\Err(\xi))
\qquad (\text{when } t < R - H).
\end{align*}

As we choose $r_j \geq 0$, $F^-$ does not contribute any terms as long as $R - H > h$.  Thus,
\begin{multline*}
\sum_{\substack{p\,\mathrm{prime}\\p/N\in E}}
\log{p} \sum_{\lvert r_j-R \rvert \le H} \varepsilon_j a_j(p)
\\ = \sum_{j>0} F(r_j/2 \pi)
\sum_{\substack{p\text{ prime}\\ p/N \in E}} \log p \ \epsilon_j a_j(p)
+ \sum_{j>0} \big(
\1_{[R - H, R + H]}(r_j)
- F(r_j/2 \pi) \big)
\sum_{\substack{p\,\mathrm{prime}\\p/N\in E}}
\log p \ \epsilon_j a_j(p)
\\ = \sum_{\substack{p\text{ prime }\\ p/N \in E}} \log p 
\sum_{j>0} F(r_j/2 \pi)\epsilon_j a_j(p)
+ \sum_{j > 0} O\bigg( \exp\left(-\frac{\xi(r_j)}{5 \log^2(\xi(r_j)+2)}\right)\bigg)
\cdot  O(r_j^{1 + \varepsilon}), 
\end{multline*}
where $\xi(r_j) = \xi(r_j; R, H)$. Here, by assuming GRH for $L(s, f_j)$, we bound the inner sum over primes of the second term (tail) by $O(r_j^{1+\varepsilon})$ for each $j>0$.
It remains to bound the sum over $j$ in the error term. By Weyl's law, the number of $j$ with $\xi(r_j) \leq T$ is $O((R Th)^{1 + \varepsilon}$. Hence, the error term is bounded by
\[\sum_{j > 0} \exp\bigg(-\frac{\xi(r_j)}{5 \log^2(\xi(r_j)+2)}\bigg) r_j^{1 + \varepsilon}
\ll (Rh)^{1 + \varepsilon} 
\int_0^\infty x^{1 + \epsilon} 
\exp\bigg(\frac{-x}{5\log^2 (x + 2)}\bigg) \, dx
\ll R^{1 + \varepsilon} h.\]
\end{proof}

\subsection{Simplifying the Trace Formula}\label{subsec:simplifyTF}

In this section, we use the trace formula with the test function $F$ constructed in the previous section to compute
\begin{equation}\label{e:Sigma_def}
\Sigma:=  \sum_{\substack{p\,\mathrm{prime}\\p/N\in E}}\log{p}\sum_{j > 0} F(r_j/2 \pi) \epsilon_j a_j(p).
\end{equation}
We begin by showing that only the hyperbolic terms of the trace formula \eqref{eq:traceformula} contribute to the main term.

\begin{lemma}\label{hyperbolicterms}
\begin{equation}\label{e:Sigma}
\Sigma = \sum_{\substack{p\text{ prime}\\ p/N \in E}}\log p 
\underset{\substack{t \in \ZZ\\ D=t^2+4p\\ t \neq \pm (p-1)}}{\sum} L(1,\psi_D) G \left( \log \left( \frac{(|t| + \sqrt{D})^2}{4 p} \right)\right)
+
O\left(\frac{R^{3 + \eps}}{h}\right).
\end{equation}
\end{lemma}

\begin{proof}
Using that $G$ has support $[-1/h, 1/h]$ that goes to $0$, for a fixed large prime $p$, \eqref{eq:traceformula} says that $\sum_{j > 0} F(r_j/2 \pi) \epsilon_j a_j(p)$ is given by
\begin{align}
\underset{\substack{t \in \ZZ\\D=t^2+4p\\ t \neq \pm (p-1)}}{\sum} L(1,\psi_D) & G \left(\log \left( \frac{(|t| + \sqrt{D})^2}{4 p} \right)\right)    
-\frac{p + 1}{\sqrt{p}} F\left(\frac{i}{4 \pi}\right) 
\nonumber 
\\ & + \int_{\log{p}}^{\infty} G(u) \cdot \frac{\sinh(u/2)}{\cosh(u/2) +  \cosh((\log p)/2) }\,du  
+ \int_{0}^{\infty} \frac {G(u + \log{p}) - G(\log{p})}{2 \sinh(u/2)}\,du \label{e:Sigma_int1}
\\ & + \int_{0}^{\infty} \frac {G(u - \log{p}) - G(\log{p})}{2 \sinh(u/2)}\,du \label{e:Sigma_int2}
\\ & + O\bigg(\sum_{m\in [\sqrt{p} e^{-\frac{1}{2h}}, \sqrt{p} e^{\frac{1}{2h}}]\cap \ZZ} \frac{\log m}{m}\bigg). \label{e:Sigma_para}
\end{align}
For \eqref{e:Sigma_para}, we get
\[\sum_{m\in [\sqrt{p} e^{-\frac{1}{2h}}, \sqrt{p} e^{\frac{1}{2h}}]\cap \ZZ} \frac{\log m}{m}
= O((\log p)^2).\]
The integrals in \eqref{e:Sigma_int1} 
asymptotically vanish for $p/N \in E$, since $p \to \infty$ while the support of $G$ goes to $0$.
These also imply that the integral \eqref{e:Sigma_int2} can be bounded by
\[\int_{0}^{\infty} \frac {G(u - \log{p}) - G(\log{p})}{2 \sinh(u/2)}\,du
\ll \int_{-1/h}^{1/h} \frac {G(u ) }{2 \sinh(u/2 + \log p/2)}\,du
\ll 1.\]

Finally, note that
\[F\left(\frac{i}{4 \pi}\right)
= \int_{-\frac{1}{h}}^{\frac{1}{h}} G(u) e^{-u/2} du
\ll \frac{1}{h}\]
as $G$ is absolutely bounded independently of $h$ or $H$.
Thus
\[\frac{p + 1}{\sqrt{p}} F(i/4 \pi)
\ll \frac{\sqrt{p}}{h}
\ll \frac{R}{h}\]
for $p/N \in E$.
Summing over the $\asymp R^2$ primes $p$ with $p/N \in E$ and noting that $h \ll R$ gives the claimed error in \eqref{e:Sigma}.
\end{proof}

Next, we show that we can replace $\log\big((|t| + \sqrt{D})^2/4p\big)$ with its first-order approximation.

\begin{lemma}\label{lemma:firstorder}
We have
\[
\Sigma
= \sum_{t \in \ZZ}
\sum_{\substack{p\,\mathrm{prime}\\4 \pi^2p/R^2\in E}}
2 \log{p}\sqrt{p} L(1,\psi_{t^2+4p})    
\frac{\cos\left(R\frac{t}{ \sqrt{p}}\right) \sin\left(H\frac{t}{\sqrt{p}}\right)}{\pi t}
\widehat{W}\left( h\frac{t}{\sqrt{p}}\right)
+ O \left(\frac{ R^{4 + \eps}}{h^3} + \frac{R^{3 + \eps}}{h}\right).
\]
\end{lemma}

\begin{proof}
For each $p$ with $p/N \in E$, denote the inner sum of \eqref{e:Sigma} in Lemma~\ref{hyperbolicterms} by
\[S_p := \sum_{\substack{t \in \ZZ \\ D=t^2+4p \\ t \neq \pm (p - 1)}} \!\!\! L(1,\psi_D) G \left( \log \left( \frac{(|t| + \sqrt{D})^2}{4 p} \right)\right).\]
By \eqref{e:G_choice}, we have 
\begin{multline*}
= \sum_{\substack{t \in \ZZ \\ D=t^2+4p\\ t \neq \pm (p - 1)}} 
L(1,\psi_D)\frac{2\cos\left(R  \log \left( \frac{(|t| + \sqrt{D})^2}{4 p} \right)\right) 
\sin\left(H  \log \bigg( \frac{(|t| + \sqrt{D})^2}{4 p} \bigg)\right)}
{\pi  \log \left( \frac{(|t| + \sqrt{D})^2}{4 p} \right)} 
\widehat{W}\bigg( h  \log \left( \frac{(|t| + \sqrt{D})^2}{4 p} \right)\bigg).
\end{multline*}
Note that
\[ \log \left( \frac{(|t| + \sqrt{D})^2}{4 p} \right) = 2 \arcsinh\left(\frac{t}{2 \sqrt{p}}\right). \]
Using the approximation $\arcsinh(x) = x+O(x^3)$ as $x\to 0$ and letting $s:= \frac{t}{2\sqrt{p}}$, we get that
\begin{align*}
S_p 
= \sum_{\substack{t \in \ZZ\\ D=t^2+4p \\ t \neq \pm (p - 1)}}
L(1,\psi_D) 
\frac{\cos\left(2 Rs + O(R s^3)\right)  
\sin\left( 2H s + O(H s^3)\right)}
{\pi s  (1 + O(s^2))} \widehat{W}\bigl(2h \, \arcsinh(s)\bigr).
\end{align*}
Now, using that
\begin{equation*}
    \frac{\cos(2Rs + Rs^3) - \cos(2Rs)}{Rs^3} = O(1) \implies
    \cos(2Rs + O(Rs^3)) = \cos(2Rs) + O(Rs^3)
\end{equation*}
(and similarly for $\sin(x)$), as well as the fact that
\begin{equation*}
  L(1,\psi_D)
  \widehat{W}\bigl(2 h \, \arcsinh(s) \bigr)
  \ll R^\eps,
\end{equation*} 
we see that
\begin{align*}
S_p
=\sum_{\substack{t \in \ZZ \\ t \neq \pm (p - 1)}} L(1,\psi_D) \frac{\cos\left( 2Rs\right) \sin\left( 2H s \right)}{\pi s}
\widehat{W}\bigl(2h \, \arcsinh(s) \bigr)
+ O\Bigl(\sum_{t \ll R/h} \bigl(s^2 H R^\eps + R^{1 + \eps}  s^2\bigr)\Bigr)
\end{align*}
since $\widehat{W}$ truncates the sum over $t$ to $\lvert t \rvert \leq 2\sqrt{p} \sinh \frac{1}{2h} \ll \frac{\sqrt{p}}{h} \ll \frac{R}{h}$.
As $R > H$, the latter error term dominates.

Finally, since $\widehat{W}$ is Schwartz, monotone, and continuous, it is differentiable and $\widehat{W}' \ll 1$.
Hence
\[\widehat{W}\left(2 h \ \arcsinh(s)\right) - \widehat{W}(2 h s) \ll h s^3\]
and
\begin{equation*}
  S_p
  =
  \sum_{\substack{t \in \ZZ \\ t \neq \pm (p - 1)}}
  L(1,\psi_D)
  \frac{\cos\left( 2Rs\right) \sin \left( 2H s \right)}{\pi s}
  \widehat{W}\left(2 hs\right)
  +
  O
  \Bigl(
    \sum_{t \ll R/h}
    \bigl(
      R^\epsilon h s^2 + R^{1 + \epsilon} s^2
    \bigr)
  \Bigr).
\end{equation*}
As $h \ll R$, the latter term again dominates.
We bound the contribution of these errors to $\Sigma$ by summing over $t$ and over the prime interval, getting
\begin{equation*}
 \sum_{\substack{p\,\mathrm{prime}\\p/N\in E}}\log{p}
 \underset{\substack{t \ll R/h }}{\sum}
 R^{1 + \eps}  s^2
 \ll
 R^{3 + \eps} \frac{1}{R^2} \frac{R^3}{h^3}\ll \frac{ R^{4+ \eps}}{h^3}.
\end{equation*}
In total, this shows that
\begin{equation*}
  \sum_{\substack{p \, \textup{prime} \\ p/N \in E}}
  (\log p) S_p
  =
  \sum_{\substack{p \, \textup{prime} \\ p/N \in E}}
  \sum_{\substack{t \in \ZZ \\ t \neq \pm (p - 1)}}
  L(1,\psi_D)
  \frac{\cos\left( 2Rs\right) \sin \left( 2H s \right)}{\pi s}
  \widehat{W}\left(2 hs\right)
  +
  O(R^{4 + \eps}/h^3).
\end{equation*}

It remains to notice that since the summand in this expression is bounded by $R^\eps H$, we can replace $N$ by $R^2/4 \pi^2$ via the asymptotic \eqref{conductor}.
This introduces a non-dominant error of size $O(R^\eps H)$.
Including the summands with $t = \pm (p-1)$ and recalling the limited support of $\widehat{W}$ introduces another non-dominant error term of size $O(R^{1+\eps} H/h)$, but allows the sum to be extended over all $t \in \mathbb{Z}$.
Plugging in $s = \frac{t}{2 \sqrt{p}}$, rearranging, and carrying the other error term from Lemma~\ref{hyperbolicterms} gives the claimed formula.
\end{proof}

\section{Character Averaging}\label{sec:L1averaging}
In this section we approximate the $L$-functions in $\Sigma$ \eqref{e:Sigma_def} by their average values. 
Define
\[\overline{\psi}_t(m)
:= \frac{1}{\phi(m^2)}
\sum_{\substack{n \bmod m^2 \\ (n, m) = 1}} \psi_{t^2 + 4 n} (m)
= \frac{1}{\phi(m^2)} \sum_{\substack{n \bmod m^2 \\ (n, m) = 1}} \psi_{t^2 - 4 n} (m) \]
and define
\[L(s, \overline{\psi}_t)
:= \sum \overline{\psi}_t(m) m^{-s}\]
for $\Re(s) > 1$.

\begin{lemma}\label{avglemma}
Assume GRH for Dirichlet $L$-functions.
Let $t \in \ZZ$, $A, B \in \RR$, and let $\Phi \in C^1[A, B]$.
Set $M:= \max_{[A, B]} \lvert \Phi(u)\rvert$ and $V:= \int_A^B \lvert \Phi'(u)\rvert du$.
Then for all $\delta \in (0, 1/10]$,
\[\sum_{\substack{n \in [A, B] \\ n \; \textup{prime}}}
L(1, \psi_{t^2 + 4 n}) \Phi(n) \log n
= L(1, \overline{\psi}_t) \int_A^B \Phi(u) du
+ O_\eps\bigl(M^{4/5} (M + V)^{1/5} B^{9/10 + \delta}\bigr).\]
\end{lemma}

\begin{proof}
The proof follows from \cite[Lemmas 4.1--4.5]{bober2023murmurations}, by replacing $-4n$ by $+ 4n$.
\end{proof}

\begin{lemma}\label{finalintegralform}
We have
\[\Sigma = 4 \int_{\frac{R - H}{2 \pi}}^{\frac{R+H}{2\pi}}
u^2 \sum_{t \in \mathbb{Z}} L(1,\overline{\psi}_t)
\int_{ \lambda_u\cdot E^{-\frac{1}{2}}} \cos\left( 2 \pi \alpha t \right)  \widehat{W}\left(t/\alpha_u\right)\, \frac{d \alpha}{\alpha^3}\, du 
+ O\left(\frac{R^{4 + \eps}}{h^3} + \frac{R^{3 + \eps}}{h}
+ \frac{R^{3 + \eps}}{h^{1/5}}\right),\]
where
\[\lambda_u = \frac{2 \pi u}{ R}, \quad \alpha_u = \frac{u}{\alpha h},\]
and where if $E = [\alpha_2^{-2}, \alpha_1^{-2}]$, then $E^{-1/2} = [\alpha_1, \alpha_2]$.
\end{lemma}

\begin{proof}
We apply Lemma~\ref{avglemma} to $\Sigma$, as given by Lemma~\ref{lemma:firstorder}, with
\[\Phi(u) 
= \frac{2}{\pi}\cos\left(R\frac{t}{ \sqrt{u}}\right)
\frac{\sin\left(H\frac{t}{\sqrt{u}}\right)}{t/\sqrt{u}}
\widehat{W}\Bigl(h\frac{t}{\sqrt{u}}\Bigr)\]
and $[A, B] = \frac{R^2}{4 \pi^2} \cdot E$.
Note that $\sqrt u \approx R$ here, and the support of $\widehat{W}$ shows the sum over $t$ only includes terms with $\lvert t \rvert \ll R/h$.
We compute
\[\cos\Bigl(R\frac{t}{\sqrt{u}}\Bigr) \ll 1,
\quad 
\frac{   \sin\bigl( H\frac{t}{ \sqrt{u}}\bigr)}{t/\sqrt{u} } \ll \min\{R/t, H\},
\quad \text{ and }\quad 
\widehat{W}\Bigl( h\frac{t}{\sqrt{u}}\Bigr) \ll 1.\]
Thus for this choice of $\Phi(u)$, 
\[M = \max_{\frac{R^2}{4\pi^2}\cdot E} |\Phi(u)| \ll \min \{R/t, H\}.\]
Now we estimate the derivatives.
Firstly,
\[\frac{\partial}{\partial u}\cos\Bigl(R\frac{t}{ \sqrt{u}}\Bigr) \ll \frac{1}{u^{3/2}} Rt \ll \frac{t}{R^2}\]
and
\[\frac{\partial}{\partial u}\widehat{W}\Bigl( h\frac{t}{\sqrt{u}}\Bigr) \ll \frac{1}{u^{3/2}} h t \ll \frac{th}{R^3} \ll \frac{t}{R^2}. \]
A simple application of the product rule shows that
\[\frac{\partial}{\partial u}\frac{ \sin\left( H\frac{t}{ \sqrt{u}}\right)}{t/\sqrt{u} } 
\ll \frac{\sqrt{u}}{t} \frac{1}{u^{3/2}} H t \ll \frac{H}{R^2}.  \] 
On the other hand, since $\sin x/x$ has a bounded derivative,
\[\frac{\partial}{\partial u}\frac{ \sin\left( H\frac{t}{ \sqrt{u}}\right)}{t/\sqrt{u} } 
= H \frac{\partial}{\partial u}\frac{ \sin\left( H\frac{t}{ \sqrt{u}}\right)}{H t/\sqrt{u} } 
\ll H \cdot \frac{1}{u^{3/2}} Ht \ll H^2 \frac{t}{R^3}. \]
To summarize, on this interval
\[\Phi'(u) \ll   \min\left\{\frac{1}{R}, \frac{t H}{R^2}\right\} 
+ \min\left\{\frac{H}{R^2}, t \frac{H^2}{R^3}\right\}.\]
The second minimum is always smaller than the first since both entries are multiplied by $H/R$.
As the length of the integral is $\asymp_E R^2$, we find $V \ll \min\{R, tH\}$.
In the range $t < R/H$, we have $M \ll H, V \ll t H$, giving
\[\sum_{t < R/H} B^{\frac{9}{10} + \delta} M^{\frac{4}{5}} (M + V)^{\frac{1}{5}} 
= R^{\frac{9}{5} + \delta} \sum_{t < R/H} H t^{\frac{1}{5}}
= \frac{R^{3 + \delta}}{H^{\frac{1}{5}}}.\]
When $t > R/H$, we have $M \ll \frac{R}{t}$, $V \ll R$, giving
\[\sum_{t \ll R/h} B^{\frac{9}{10} + \delta} M^{\frac{4}{5}} (M + V)^{\frac{1}{5}} 
= R^{\frac{9}{5} + \delta} \sum_{t \ll R/h} R t^{-\frac{4}{5}} = \frac{R^{3 + \delta}}{h^{\frac{1}{5}}}.\]
No other ranges of $t$ contribute due to the support of $\widehat{W}$.
Hence, for any $\delta \in (0, 1/10]$, we can approximate the main term for $\Sigma$ in Lemma \ref{lemma:firstorder}
\[\sum_{t \in \mathbb{Z}}
\sum_{\substack{p\,\mathrm{prime}\\4 \pi^2p/R^2\in E}}2 \log{p}\sqrt{p}
L(1,\psi_D) \frac{ \cos\left(R\frac{t}{ \sqrt{p}}\right)  \sin\left( H\frac{t}{ \sqrt{p}}\right)}{\pi t } \widehat{W}\left( h\frac{t}{\sqrt{p}}\right),\]
by
\[\sum_{t \in \mathbb{Z}} 2 L(1,\overline{\psi}_t)
\int_{\frac{R^2}{4 \pi^2}\cdot E}
\frac{ \cos\bigl(R\frac{t}{ \sqrt{x}}\bigr)  \sin\bigl( H\frac{t}{ \sqrt{x}}\bigr)}
{\pi t /\sqrt{x}}
\widehat{W}\Bigl(h\frac{t}{\sqrt{x}}\Bigr) dx
+ O\Bigl( \frac{R^{3 + \delta}}{H^{1/5}} + \frac{R^{3 + \delta}}{h^{1/5}}\Bigr).\]

Finally, we use that $2 \cos (Ry) \sin(Hy)/(\pi y)$ is the Fourier transform of $\1_{[\pm \frac{R-H}{2\pi}, \pm \frac{R+H}{2\pi}]}(u)$ 
(and collect the two exponentials into a cosine) to rewrite this as
\[2\int_{[\frac{R - H}{2\pi}, \frac{R+H}{2\pi}]}
\underset{t \in \ZZ}{\sum}  L(1,\overline{\psi}_t)
\int_{\frac{R^2}{4 \pi^2}\cdot E} \cos\left(2 \pi u \frac{ t}{ \sqrt{x}}\right) 
\widehat{W}\left( h\frac{t}{\sqrt{x}}\right) \, dx \, du. \]
By changing a variable $x$ to $\frac{1}{x^2}$, 
and then changing $x$ to $\frac{\alpha}{u}$, we get
\begin{multline*}
2\int_{[\frac{R - H}{2\pi}, \frac{R+H}{2\pi}]}
\underset{t \in \ZZ}{\sum} L(1,\overline{\psi}_t)
\int_{\frac{2\pi}{R} \cdot E^{-\frac{1}{2}}}\frac{2}{x^3}\cos\left(2 \pi u t x\right) \widehat{W}\left( h t x\right) \, dx \, du
\\= 4\int_{[\frac{R - H}{2\pi}, \frac{R+H}{2\pi}]}
u^2 \underset{t \in \ZZ}{\sum} L(1,\overline{\psi}_t)
\int_{\lambda_u \cdot E^{-\frac{1}{2}}}
\cos\left(2 \pi \alpha t \right)
\widehat{W}\left(\frac{t}{\alpha_u}\right) \, \frac{d\alpha}{\alpha^3} \, du,
\end{multline*}
with $\lambda_u = \frac{2 \pi u}{R}$ and $\alpha_u = \frac{u}{ \alpha h}$.
\end{proof}

\section{Circle Method}\label{subsec:circlemethod}

In this section we complete the computation of the numerator by breaking up the inner integral into minor and major arcs.

We have the following result from \cite[Proposition 5.1]{bober2023murmurations}. 
(It is stated with stronger hypotheses on $W$ in \cite{bober2023murmurations}, but the proof applies equally well to any even Schwartz function $W$ with $\widehat{W}(0)=1$.)
\begin{proposition}[\cite{bober2023murmurations}, Proposition $5.1$]\label{prop:5.1}
Assume GRH for Dirichlet $L$-functions.
Let $W: \RR\to \RR$ be an even Schwartz function with $\widehat{W}(0) = 1$.
Let $\alpha, \theta, x \in \mathbb{R}$ and $a, q \in \mathbb{Z}$ with $x, q \geq 1$, $\gcd(a, q) = 1$, $\alpha = a/q + \theta$ and $|\theta| < 1/q^2$.
Then for all $\eps > 0$,
\begin{multline}
\sum_{t \in \ZZ} L(1, {\overline{\psi}}_t) \cos(2 \pi \alpha t) \widehat{W}\Bigl(\frac{t}{x} \Bigr)
= \frac{\mu(q)^2}{\phi(q)^2 \sigma(q)} x W(\theta x)
\\ + O\left( q x^{-1} \max\{1, |\theta|x\} \right)
+ O_\eps \bigl(q^3 x^{-\frac{7}{4} + \eps} \max \{1, |\theta|x\}^{\frac{7}{2}} \bigr).
\end{multline}
\end{proposition}

\begin{lemma}\label{lem:innerintegral}
Let $u \in  [\frac{R - H}{2\pi}, \frac{R + H}{2\pi}]$, 
$\lambda_u = \frac{2 \pi u}{ R}$, 
$x_u(\alpha) = \frac{u}{\alpha h}$, 
and $I_u = \lambda_u\cdot E^{-1/2}= [\lambda_u \alpha_1, \lambda_u\alpha_2]$. Then
\begin{multline*}
4\int_{\frac{R-H}{2 \pi}}^{\frac{R+H}{2 \pi}} u^2 \sum_{t \in \mathbb{Z}} L(1,\overline{\psi}_t)
\int_{I_u} \cos\left( 2 \pi \alpha t \right) \widehat{W}\left(\frac{t}{x_u(\alpha)}\right)\, \frac{d \alpha}{\alpha^3} \,du
\\ = \bigg(\frac{R^2 H}{\pi^3} +O(RH^2)\bigg) \sum_{\frac{a}{q} \in  E^{-1/2}}^* \frac{\mu(q)^2}{ \phi(q)^2 \sigma(q)}\left(\frac{a}{q}  \right)^{-3}
\\ +   R^2H \cdot O \left(  R^\eps\left(\frac{h^2}{RH} +  \frac{H^{\frac{3}{2}}}{h R^{\frac{1}{2}}} 
+ \frac{\sqrt{H}}{\sqrt{R}} + \frac{H^{\frac{3}{2}}}{h^{\frac{5}{4}} R^{1/4}} + \frac{h^2}{H^{\frac{3}{2}} \sqrt{R}} 
+ \frac{h^{\frac{11}{4}}}{H^{\frac{5}{2}} R^{\frac{1}{4}}} \right) \right)
+ o_E(1). 
\end{multline*}
Here the $*$ in the sum over $E^{-1/2}$ indicates that the terms occurring at the endpoints of $E$ contribute half.
\end{lemma}

\begin{proof}
Let $P = \frac{\sqrt{R/H}}{(\log R)^{10}}$ and $Q = \frac{\sqrt{RH}}{h}$,  
so that
$\frac{x_u(\alpha)}{PQ} \asymp (\log R)^{10}$.
For $a, q\in \ZZ$ with $q>0$ and $\gcd(a, q)=1$, define
\[\frakM\left(\frac{a}{q}\right)
:= \left \{\frac{a}{q} + \theta : |\theta| < \frac{1}{q Q} \right\}. \]
Note that for $a, a', q, q'\in \ZZ$ with $0< q, q' \leq P$ and $\gcd(a', q')=\gcd(a,q)=1$, 
if $\frac{a}{q}\neq \frac{a'}{q'}$, 
\[\frac{1}{P\min\{q, q'\}} \leq \frac{1}{qq'}\leq \left|\frac{a}{q}-\frac{a'}{q'}\right|.
\]
Since $\frac{P}{Q} = \frac{h}{(\log R)^{10}}< \frac{1}{2}$ because $h=o(H)$, we have $\frakM\left(\frac{a}{q}\right)\cap \frakM\left(\frac{a'}{q'}\right)=\emptyset$.
We define 
\[\frakM_u :=
I_u \cap \bigcup_{\substack{a, q \in \ZZ \\ 0 < q \leq P\\ \gcd(a, q) = 1}} \frakM\left(\frac{a}{q}\right)
\quad\text{ and } \quad \frakm_u:= I_u \setminus \frakM_u.\]

Note that 
\[|\lambda_u-1| = \frac{|2\pi u-R|}{R} \leq \frac{H}{R} \in (R^{-\frac{1}{6}+\delta}, R^{-\delta}).\]
Thus, as $u$ varies, the endpoint $\lambda_u \alpha_i$ (for $i=1, 2$) is confined to the interval $[(1-\frac{H}{R})\alpha_i, (1+\frac{H}{R})\alpha_i]$.
By Dirichlet's theorem, we can choose a fraction $\frac{a_i}{q_i}$ such that $q_i \leq 3P$ and 
\[\alpha_i= \frac{a_i}{q_i} + \theta_i\quad\text{ with } |\theta_i|\leq \frac{1}{q_i3P}.\]
Assume that $R$ is sufficiently large to ensure that $a_i>0$.

Let $\frac{a}{q}$ be a fraction with $0< q\leq P$ and $\frac{a}{q}\neq \frac{a_i}{q_i}$ for both $i=1, 2$.
Adding the inequalities
\[qq_i \frac{H}{R} \alpha_i \leq 3P^2 \frac{H}{R} \alpha_2 = \frac{3\alpha_2}{(\log R)^{20}}, \]
\[\frac{q_i}{Q} \leq \frac{3P}{Q}=\frac{3\sqrt{R/H}}{(\log R)^{10}} \frac{h}{\sqrt{RH}} = \frac{3h}{(\log R)^{10} H}\]
and
\[qq_i|\theta_i|\leq \frac{q}{3P}\leq \frac{1}{3}, \]
for sufficiently large $R$, we have 
\[qq_i|\theta_i| + \frac{q_i}{Q}+qq_i \frac{H}{R} \alpha_i < 1.\]
Dividing by $qq_i$ for both sides, we get
\[|\theta_i| + \frac{1}{qQ} + \frac{H}{R}\alpha_i < \frac{1}{qq_i}\leq \left|\frac{a}{q}-\frac{a_i}{q_i}\right|.\]
So we have 
\[\left|\frac{a}{q}-\alpha_i\right| = \left|\frac{a}{q}-\frac{a_i}{q_i} -\theta_i\right|\geq \frac{1}{qQ}+\frac{H}{R}\alpha_i. \]
Thus, 
\[\frakM\left(\frac{a}{q}\right)\cap \left[\alpha_i\left(1-\frac{H}{R}\right), \alpha_i\left(1+\frac{H}{R}\right)\right]= \emptyset.\]
Therefore, recalling that $E=[\alpha_2^{-2}, \alpha_1^{-2}]$, we have 
\[
\frakM\left(\frac{a}{q}\right)\cap I_u = \begin{cases}
\frakM\left(\frac{a}{q}\right) & \text{ if } \left(\frac{a}{q}\right)^{-2}\in E, \\
\emptyset & \text{ if } \left(\frac{a}{q}\right)^{-2}\notin E. 
\end{cases}
\]

We split the integral over $I_u$ as 
\begin{multline}\label{e:intIu_maj+min+ends}
\int_{I_u} \sum_{t\in \ZZ} L(1, \overline{\psi_t}) 
\cos(2\pi \alpha t) \widehat{W}\left(\frac{t}{x_u(\alpha)}\right) \, \frac{d\alpha}{\alpha^3}
= \int_{\frakm_u} + \int_{\frakM_u}
\\ = \int_{\frakm_u} + \sum_{\substack{\frac{a}{q}\in (\QQ\setminus\{\frac{a_1}{q_1}, \frac{a_2}{q_2}\}) \cap [\alpha_1, \alpha_2]\\ 0<q\leq P}}
\int_{\frakM\left(\frac{a}{q}\right)} + \int_{\frakM_u \cap \frakM\left(\frac{a_1}{q_1}\right)}
+ \int_{\frakM_u\cap \frakM\left(\frac{a_2}{q_2}\right)}.
\end{multline}
Note that $\frakM_{u} \cap \frakM\left(\frac{a_i}{q_i}\right)=\emptyset$ if $q_i>P$.

We evaluate the terms of the RHS in \eqref{e:intIu_maj+min+ends} by using Proposition \ref{prop:5.1}. 
By Dirichlet's theorem, for $\alpha\in I_u$ we may choose $q\in \ZZ_{\geq 1}$ with $q\leq Q$ and $a\in \ZZ$ with $\gcd(a, q)=1$ such that 
\[\left|\alpha-\frac{a}{q}\right|\leq \frac{1}{qQ}.\]
If $\alpha\in \frakm_{u}$, then $q>P$, in which case 
\[\left|\alpha-\frac{a}{q}\right| x_u(\alpha)
<\frac{u}{\alpha h}\frac{1}{PQ} 
\ll_E \frac{(R+H)(\log R)^{10}}{R} \ll_E (\log R)^{10},\]
so by Proposition \ref{prop:5.1}, 
\begin{multline}\label{e:minoracr_error}
\sum_{t\in \ZZ} L(1, \overline{\psi_t}) \cos(2\pi \alpha t) \widehat{W}\left(\frac{t}{x_u(\alpha)}\right)
\ll_{E,\varepsilon} \frac{x_u(\alpha)}{P^{3-\varepsilon}} 
+ Q(x_u(\alpha))^{-1} (\log R)^{10} + Q^3 (x_u(\alpha))^{-\frac{7}{4}+\varepsilon} 
(\log R)^{35}
\\ \ll_{E, \varepsilon} 
\frac{u}{h} (\log R)^{10(3-\varepsilon)} \left(\frac{R}{H}\right)^{-\frac{3}{2}+\varepsilon} 
+ \frac{\sqrt{RH}}{u} (\log R)^{10}
+ (RH)^{\frac{3}{2}+\varepsilon} h^{-\frac{5}{4}-\varepsilon} u^{-\frac{7}{4}+\varepsilon}
(\log R)^{35}
\\ \ll_{E, \varepsilon} R^{-\frac{1}{2}+\varepsilon} H^{\frac{3}{2}+\varepsilon} h^{-1}
+ R^{-\frac{1}{2}+\varepsilon} H^{\frac{1}{2}} 
+ R^{-\frac{1}{4}+\varepsilon} H^{\frac{3}{2}+\varepsilon} h^{-\frac{5}{4}-\varepsilon}
\ll_{E, \varepsilon} R^{-\frac{1}{4}+\varepsilon} H^{\frac{3}{2}} h^{-\frac{5}{4}}.
\end{multline}

Note that the above estimate also applies to the error terms in Proposition \ref{prop:5.1} for $\alpha\in \frakM\left(\frac{a}{q}\right)$, a major arc, with $\left|\alpha-\frac{a}{q}\right|x_u(\alpha)\leq 1$: 
\[\bigg|\sum_{t\in\ZZ} L(1, \overline{\psi_t}) \cos(2\pi \alpha t) \widehat{W}\left(\frac{t}{x_u(\alpha)}\right)- \frac{\mu(q)^2}{\varphi(q)^2\sigma(q)} x_u(\alpha) W(\theta x_u(\alpha))\bigg|
\ll_{E, \varepsilon} R^{-\frac{1}{4}+\varepsilon} H^{\frac{3}{2}} h^{-\frac{5}{4}}.\]
Here $\theta=\alpha-\frac{a}{q}$.
For $\alpha\in \frakM\left(\frac{a}{q}\right)$ with $\left|\alpha-\frac{a}{q}\right|x_u(\alpha) >1$, the error term is 
\begin{multline*}
\bigg|\sum_{t\in\ZZ} L(1, \overline{\psi_t}) \cos(2\pi \alpha t) \widehat{W}\left(\frac{t}{x_u(\alpha)}\right)- \frac{\mu(q)^2}{\varphi(q)^2\sigma(q)} x_u(\alpha) W(\theta x_u(\alpha))\bigg|
\\ \ll_{\varepsilon, E} Q^{-1} 
+ q^3 (uh^{-1})^{\frac{7}{4}+\varepsilon} (qQ)^{-\frac{7}{2}}
\ll Q^{-1} + q^{-\frac{1}{2}} Q^{-\frac{7}{2}}  (Rh^{-1})^{\frac{7}{4}+\varepsilon}. 
\end{multline*}
Summing over all arcs ($\ll q$ for each $q$) and accounting for their length ($\ll \frac{1}{qQ}$) gives the error term
\[R^\eps \left(P Q^{-2} + P^{\frac{1}{2}} Q^{-\frac{9}{2}} (R/h)^{\frac{7}{4}}\right) 
\ll R^{\eps} \left(\frac{h^2}{H^{\frac{3}{2}} \sqrt{R}} 
+ \frac{h^{\frac{11}{4}}}{H^{\frac{5}{2}} R^{\frac{1}{4}}} \right).\]

Now we investigate the main term in Proposition \ref{prop:5.1} on the major arcs.
Note that for $\alpha\in I_u$, we have 
\[x_u(\alpha)\geq x_u(\lambda_u \alpha_2)=\frac{R}{2\pi h}\frac{1}{\alpha_2}= PQ \frac{(\log R)^{10}}{2\pi\alpha_2}.\]
Therefore, when $\left(\frac{a}{q}\right)^{-2}\in E$ and $I_u\cap \frakM\left(\frac{a}{q}\right) = \frakM\left(\frac{a}{q}\right)$ (i.e., $\frac{a}{q}\neq \frac{a_i}{q_i}$ for $i=1, 2$), we have 
\[
\left\{x_u(\alpha)\left(\alpha-\frac{a}{q}\right):\, \alpha\in \frakM\left(\frac{a}{q}\right)\right\}
\supseteq [-(\log R)^{10}, (\log R)^{10}].
\]
By changing the variable $x=x_u(\alpha)(\alpha-\frac{a}{q})$, we can approximate 
\begin{multline*}
\bigg|\int_{\frakM\left(\frac{a}{q}\right)} x_u(\alpha) W\left(\left(\alpha-\frac{a}{q}\right) x_u(\alpha)\right)\, \frac{d\alpha}{\alpha^3}
- \left(\frac{a}{q}\right)^{-3} \int_{\RR} W(x)\left(1-\frac{h}{u}x\right)^2\, dx\bigg|
\\ \leq \left(\frac{a}{q}\right)^{-3} \int_{|x|>(\log R)^{10}} W(x)\left(1-\frac{h}{u}x\right)^2\, dx
\ll P^3 e^{-(\log R)^{9}}.
\end{multline*}
Summed over all fractions constituting the major arcs that overlap $I_u$ (of which there are $\ll P^2$, since $|I_u| \asymp 1$), 
the cumulative error is bounded by
\[\ll P^2 \cdot P^3 e^{(\log P)^{9}} \ll R^{-\log R}.\]
On the other hand,
\[\left( \frac{a}{q}\right)^{-3} \int_{\RR} W(x) \left(1 - \frac{x h}{u}\right)^2\, dx  
= \left( \frac{a}{q}\right)^{-3} \left( 1 +   \frac{h^2}{u^2} \int_{\RR} W(x) x^2 \, dx\right) 
=  \left( \frac{a}{q}\right)^{-3} + O\left(\frac{h^2}{R^2}\right), \]
and the error term summed over all the rationals gives
\[\frac{h^2}{R^2} P^2 \ll R^{\eps} \frac{h^2 }{R H}.\]

We note that the final evaluation of the integrals here do not include any dependence on $u$. Hence, the integral over $u$ will give a factor of
\[\frac{R^2H}{\pi^3} + O(RH^2).\]

Now we address major arcs that overlap $I_u$ partially: $\frakM_u\cap \frakM\left(\frac{a_i}{q_i}\right)$ for $i=1, 2$.
Consider the case when $\alpha_i$ is irrational.
Suppose that $\left(\frac{a_i}{q_i}\right)^{-2}\in E$ but
\begin{equation}\label{e:endpt_insupp}
(-1)^i x_u(\lambda_u \alpha_i) \left(\lambda_u \alpha_i - \frac{a_i}{q_i}\right) < (\log R)^{10}.
\end{equation}
Rearranging this inequality, we find
\[\left|\alpha_i-\frac{a_i}{q_i}\right| 
= \left|\lambda_u\alpha_i + \alpha_i(1-\lambda_u)-\frac{a_i}{q_i}\right|
< |\alpha_i|\bigg((\log R)^{10} \frac{2\pi h}{R}+ \frac{H}{R}\bigg) 
\ll_E \frac{H}{R}.\]
A similar calculation leads to the same inequality when $\left(\frac{a_i}{q_i}\right)^{-2}\notin E$ but satisfying \eqref{e:endpt_insupp}.
By the definition of irrationality measure, we have 
$|\alpha_i-\frac{a_i}{q_i}| \geq q_i^{-\mu_i+o(1)}$ as $R\to \infty$ when $\alpha_i$ has irrationality measure $\mu_i$, and then 
\[\frac{\mu(q_i)^2}{\varphi(q_i)^2\sigma(q_i)} \leq q_i^{-3+o(1)} \leq \left(\frac{H}{R}\right)^{\frac{3}{\mu_i}+o(1)}.\]

When $\alpha_i$ is rational, we have $\frac{a_i}{q_i}=\alpha_i$ for sufficiently large $R$.
Writing 
\[x=x_u(\alpha)(\alpha-\alpha_i) = \frac{u}{h}\left(1-\frac{\alpha_i}{\alpha}\right)\]
for $\alpha\in I_u = [\lambda_u\alpha_1, \lambda_u\alpha_2]$, we have $(-1)^i (\lambda_u \alpha_i-\alpha)\geq 0$.
Multiplying by $\frac{1}{\alpha \lambda_u}$, we get
\[0\leq (-1)^i \left(\frac{\alpha_i}{\alpha} - \frac{1}{\lambda_u}\right)
= (-1)^i \frac{h}{u} \left(-x+ \frac{2\pi u-R}{2\pi h} \right), \]
which implies that 
\[(-1)^i x \leq (-1)^i \frac{2\pi u-R}{2\pi h}.\]

Assume that $i=1$. 
Following the above arguments, we have 
\[\left\{x_u(\alpha)(\alpha-\alpha_1):\, \alpha\in \frakM(\alpha_1)\cap I_u\right\} \supseteq \left[\frac{u}{h}-\frac{R}{2\pi h}, (\log R)^{10}\right].\]
Thus, we get the full interval $[-(\log R)^{10}, (\log R)^{10}]$ when $u\in [\frac{R-H}{2\pi}, \frac{R}{2\pi}+O(h(\log R)^{10})]$.
When $u\in [\frac{R}{2\pi}+O(h(\log R)^{10}), \frac{R+H}{2\pi}]$, we have 
\[\frac{2\pi u-R}{2\pi h} = \frac{1}{h}\left(u-\frac{R}{2\pi}\right)\gg (\log R)^{10}.\]
So following the arguments for other major arcs above, we have
\begin{multline*}
\int_{\frac{R-H}{2\pi}}^{\frac{R+H}{2\pi}}
\int_{\frakM(\alpha_1)\cap I_u} x_u(\alpha) W\left(\left(\alpha-\alpha_1\right) x_u(\alpha)\right) \, \frac{d\alpha}{\alpha^3}\, 
u^2\, du
\\ = \int_{\frac{R-H}{2\pi}}^{\frac{R}{2\pi}+O(h(\log R)^{10})}
\bigg(\alpha_1^{-3} \int_{\RR} W(x)\left(1-\frac{h}{u}x\right)^2\, dx 
+ O\left( P^3 e^{-(\log R)^9}\right)\bigg)\, u^2\, du
\\ = \int_{\frac{R-H}{2\pi}}^{\frac{R}{2\pi}+O(h(\log R)^{10})}
\bigg(\alpha_1^{-3} + O\left(P^3 e^{-(\log R)^9}\right)+O\left(\frac{h^2}{R^2}\right)\bigg)\, u^2\, du.
\end{multline*}
Similar arguments work for $i=2$.
Hence, integrating over whole range of $u$, we get half the major arc contribution for the end points with relative error $O(R^{\eps}h/H)$.

\end{proof}

\begin{proof}[Proof of Theorem~\ref{mainthm}]
We combine Lemma~\ref{smoothingerror}, Lemma~\ref{finalintegralform}, and Lemma~\ref{lem:innerintegral}. 
Assuming that $h = o( H R^\eps)$ and $H =o (R^{1 - \eps})$,
\begin{align*}
\sum_{\substack{p\,\mathrm{prime}\\p/N\in E}}\log{p}
\sum_{|r_j-R|\le H}\epsilon_j a_j(p)
& = 4 \int_{\frac{R-H}{2\pi}}^{\frac{R+H}{2\pi}} 
u^2\, du 
\sum_{\frac{q^2}{a^2} \in E}^* 
\frac{\mu(q)^2}{ \phi(q)^2 \sigma(q)}\Bigl(\frac{a}{q}\Bigr)^{-3} 
\\ &\quad + R^2 H \cdot O \biggl( R^\eps \Bigl(\frac{ R^{2}}{Hh^3} + \frac{R}H {h^{\frac{1}{5}}}
+  \frac{H^{\frac{3}{2}}}{h R^{\frac{1}{2}}} 
+ \frac{H^{\frac{3}{2}}}{h^{\frac{5}{4}} R^{\frac{1}{4}}} \Bigr) \biggr).
\end{align*}

The main term evaluates to
\[\left( \frac{1 }{ \pi^3} R^2 H + O(R H^2)\right)
\sum_{\frac{q^2}{a^2} \in E}^* 
\frac{\mu(q)^2}{ \phi(q)^2 \sigma(q)}\left(\frac{a}{q}  \right)^{-3}.\]
Let $H = R^{1 - \delta_H}, h = R^{1 - \delta_h}$ for fixed $\delta_h > \delta_H > 0$; 
then the error term divided by $R^2 H$ becomes
\[R^{3 \delta_h - 2} + R^{\delta_H + \frac{\delta_h}{5} - \frac{1}{5}} 
+ R^{\delta_h - \frac{3}{2} \delta_H} 
+ R^{\frac{5}{4} \delta_h - \frac{3}{2} \delta_H}.\]
We can find a configuration of $\delta_h$, $\delta_H$ for which all of these terms are $o(1)$, for instance by taking
\[0 < \delta_H < \frac{5}{31} \quad \text{ and } \quad 
\delta_H < \delta_h < \frac{6}{5} \delta_H.\]

The only remaining part is to handle the denominator in Theorem~\ref{mainthm}.
For this, we appeal to a refined Weyl law for the count of Maass forms as in~\cite[Thm. 2]{Risager} or~\cite{bookerplattzeta}, 
which shows that
\[\sum_{r_j < T} 1 
= \frac{T^2}{12} - \frac{2 T \log T}{\pi} + \frac{1}{\pi}(2 - \log{2} + \log{\pi}) T + O\bigg(\frac{T}{\log T}\bigg).\]
This shows
\[\sum_{|r_j - R| < H} 1 
= \frac{RH}{3} - \frac{2}{\pi}((R + H) \log (R + H) - (R - H) \log (R- H))+  O \bigg(\frac{R}{\log R}\bigg)
= \frac{RH}{3} + O\!\left(\frac{R}{\log{R}}\right).\]
Assembling everything together, we compute
\begin{align*}
\frac{\sum_{\substack{p\,\mathrm{prime}\\p/N\in E}}\log{p}\sum_{|r_j-R|\le H}\epsilon_j a_j(p) }{\sum_{\substack{p\,\mathrm{prime}\\p/N\in E}}\log{p}\sum_{|r_j-R|\le H}1} 
&= \frac{R^2 H/\pi^3 + o(R^2 H)}{R H/3\sum_{ \substack{p\,\mathrm{prime}\\p/N\in E}}\log p + O(R^3)  } 
\sum_{\frac{q^2}{a^2} \in E}^* \frac{\mu(q)^2}{ \phi(q)^2 \sigma(q)}\left(\frac{a}{q}  \right)^{-3} 
\\&=\frac{R^2 H/\pi^3 }{|E| R^3 H /(12 \pi^2)  + O(R^{2 + \eps} H)} \sum_{\frac{q^2}{a^2} \in E}^* \frac{\mu(q)^2}{ \phi(q)^2 \sigma(q)}\left(\frac{a}{q}  \right)^{-3} +o(R^{-1})
\\ &=\frac{12}{\pi}\frac{1}{R}
\sum_{\frac{q^2}{a^2} \in E}^* \frac{\mu(q)^2}{ \phi(q)^2 \sigma(q)}\left(\frac{a}{q}  \right)^{-3} 
+ o(R^{-1}).
\end{align*}
Plugging in the estimate $R \approx 2 \pi \sqrt{N}$, we complete the proof.
\end{proof}

\thispagestyle{empty}
{\footnotesize
\bibliographystyle{amsalpha}
\bibliography{reference}

\providecommand{\bysame}{\leavevmode\hbox to3em{\hrulefill}\thinspace}
\providecommand{\MR}{\relax\ifhmode\unskip\space\fi MR }
\providecommand{\MRhref}[2]{%
  \href{http://www.ams.org/mathscinet-getitem?mr=#1}{#2}
}
\providecommand{\href}[2]{#2}
\begin{thebibliography}{BBLLD23}

\bibitem[BBLLD23]{bober2023murmurations}
Jonathan Bober, Andrew~R. Booker, Min Lee, and David Lowry-Duda,
  \emph{Murmurations of modular forms in the weight aspect}.

\bibitem[BL17]{AndyMin}
Andrew~R. Booker and Min Lee, \emph{The {S}elberg trace formula as a
  {D}irichlet series}, Forum Math. \textbf{29} (2017), no.~3, 519--542.
  \MR{3641663}

\bibitem[BP19]{bookerplattzeta}
Andrew~R. Booker and David~J. Platt, \emph{Turing's method for the {S}elberg
  zeta-function}, Comm. Math. Phys. \textbf{365} (2019), no.~1, 295--328.
  \MR{3900832}

\bibitem[Bum97]{bump}
Daniel Bump, \emph{Automorphic forms and representations}, Cambridge Studies in
  Advanced Mathematics, vol.~55, Cambridge University Press, Cambridge, 1997.
  \MR{1431508}

\bibitem[Hej76]{Hej1}
Dennis~A. Hejhal, \emph{The {S}elberg trace formula for {${\rm PSL}(2,R)$}.
  {V}ol. {I}}, Lecture Notes in Mathematics, vol. Vol. 548, Springer-Verlag,
  Berlin-New York, 1976. \MR{439755}

\bibitem[Hej83]{Hej2}
\bysame, \emph{The {S}elberg trace formula for {${\rm PSL}(2,\,{\bf R})$}.
  {V}ol. 2}, Lecture Notes in Mathematics, vol. 1001, Springer-Verlag, Berlin,
  1983. \MR{711197}

\bibitem[HLOP22]{HLOP}
Yang-Hui He, Kyu-Hwan Lee, Thomas Oliver, and Alexey Pozdnyakov,
  \emph{Murmurations of elliptic curves},
  \href{https://arxiv.org/abs/2204.10140}{arXiv:2204.10140}, 2022.

\bibitem[Ing34]{Ingham}
A.~E. Ingham, \emph{A {N}ote on {F}ourier {T}ransforms}, J. London Math. Soc.
  \textbf{9} (1934), no.~1, 29--32. \MR{1574706}

\bibitem[Ris04]{Risager}
Morten~S. Risager, \emph{Asymptotic densities of {M}aass newforms}, J. Number
  Theory \textbf{109} (2004), no.~1, 96--119. \MR{2098479}

\bibitem[Sel56]{Sel}
A.~Selberg, \emph{Harmonic analysis and discontinuous groups in weakly
  symmetric {R}iemannian spaces with applications to {D}irichlet series}, J.
  Indian Math. Soc. (N.S.) \textbf{20} (1956), 47--87. \MR{88511}

\bibitem[SH23]{andreithesis}
Andrei Seymour-Howell, \emph{Rigorous computation of maass cusp forms}, Ph.D.
  thesis, University of Bristol, 2023.

\bibitem[Str16]{Andreaspre}
Andreas Str\"ombergsson, \emph{Explicit trace formula for {Hecke} operators},
  Preprint (2016).

\bibitem[Sut22]{sutherland}
Andrew~V. Sutherland, \emph{Letter to {M}ichael {R}ubinstein and {P}eter
  {S}arnak}, \url{https://math.mit.edu/~drew/RubinsteinSarnakLetter.pdf}, 2022.

\bibitem[Zub23]{zubrilina}
Nina Zubrilina, \emph{Murmurations},
  \href{https://arxiv.org/abs/2310.07681}{arXiv:2310.07681}, 2023.

\end{thebibliography}
}


\end{document}